\documentclass[ctagsplt,12pt]{amsart}
\RequirePackage{newlfont}
\usepackage{amscd}
\usepackage{amsmath}
\usepackage{amsfonts}
\usepackage{amssymb}
\usepackage{enumerate}
\usepackage{color}

\begin{document}
\title[Coupled fixed points of monotone mappings]{Coupled Fixed Points of monotone mappings in a metric space with a graph}

\author{M. R. Alfuraidan, and M. A. Khamsi}

\address{Monther Rashed Alfuraidan\\Department of Mathematics \& Statistics \\
King Fahd University of Petroleum and Minerals\\
Dhahran 31261, Saudi Arabia.}
\email{monther@kfupm.edu.sa}
\address{Mohamed Amine Khamsi\\Department of Mathematical Sciences, University of Texas at El Paso, El Paso, TX 79968, USA.}
\email{mohamed@utep.edu}


\maketitle
\newtheorem{theorem}{Theorem}[section]
\newtheorem{acknowledgement}{Acknowledgement}
\newtheorem{algorithm}{Algorithm}
\newtheorem{axiom}{Axiom}
\newtheorem{case}{Case}
\newtheorem{claim}{Claim}
\newtheorem{conclusion}{Conclusion}
\newtheorem{condition}{Condition}
\newtheorem{conjecture}{Conjecture}
\newtheorem{corollary}{Corollary}[section]
\newtheorem{criterion}{Criterion}
\newtheorem{definition}{Definition}[section]
\newtheorem{example}{Example}[section]
\newtheorem{exercise}{Exercise}
\newtheorem{lemma}{Lemma}[section]
\newtheorem{notation}{Notation}
\newtheorem{problem}{Problem}
\newtheorem{proposition}{Proposition}[section]
\newtheorem{remark}{Remark}[section]
\newtheorem{solution}{Solution}
\newtheorem{summary}{Summary}
\newtheorem{property}{Property}
\begin{abstract}
In this work, we define the concept of mixed $G$-monotone mappings defined on a metric space endowed with a graph.  Then we obtain sufficient conditions for the existence of coupled fixed points for such mappings when a weak contractivity type condition is satisfied.
\end{abstract}

\section{Introduction}
\noindent Investigation of the existence of fixed points for single-valued mappings in partially ordered metric spaces was initially considered by Ran and Reurings in \cite{RR} who proved the following result:

\begin{theorem}\label{ran} \cite{RR} Let $(X,\preceq)$ be a partially ordered set such that every pair $x,y\in X$ has an upper and lower bound. Let $d$ be
a metric on $X$ such that $(X,d)$ is a complete metric space. Let $f: X\rightarrow X$ be a continuous monotone (either order preserving or
order reversing) mapping. Suppose that the following conditions hold:
\begin{enumerate}
  \item There exists $k\in [0,1)$ with
\begin{center}
$d(f(x),f(y))\leq k\ d(x,y)$, for all $x,y\in X$ such that $x \preceq y$.
\end{center}
  \item There exists an $x_{0} \in X$ with $x_{0} \preceq f(x_{0})$ or $f(x_{0}) \preceq x_0$.
\end{enumerate}
Then $f$ is a Picard Operator (PO), that is $f$ has a unique fixed point $x^{*}\in X$ and for each $x\in X$, $\lim\limits_{n\rightarrow \infty} f^{n}(x)=x^{*}$.
\end{theorem}

\noindent After this, different authors considered the problem of existence of a fixed point for contraction mappings in partially ordered metric spaces; see \cite{ BB,DMD, HS, NPRL} and references cited therein. Nieto, Pouso and Rodriguez-Lopez in \cite{NPRL} extended the ideas of \cite{RR}  to prove the existence of solutions to some differential equations.
\medskip

\noindent Generalizing the Banach contraction principle for multivalued mappings, Nadler \cite{N} obtained the following result:

\begin{theorem}Let $(X,d)$ be a complete metric space.  Denote by $CB(X)$ the set of all nonempty closed bounded subsets of $X$.  Let $F: X \rightarrow CB(X)$ be a multivalued mapping. If there exists $k\in [0,1)$ such that
$$H(F(x),F(y))\leq k\ d(x,y)$$
for all $x,y\in X,$ where $H$ is the Pompeiu-Hausdorff metric on $CB(X)$, then $F$ has a fixed point in $X$, i.e., there exists $x \in X$ such that $x \in F(x)$.
\end{theorem}

\noindent Recently, two results have appeared, giving sufficient conditions for $f$ to be a PO, if $(X,d)$ is endowed with a graph. The first result in this direction was given by Jachymski and Lukawska \cite{J,LJ} which generalized the results of \cite{DMD,NPRL, OP, PR} to single-valued mapping in metric spaces with a graph instead of partial ordering. The extension of Jachymaski's result to multivalued mappings is done in \cite{monther}.

\medskip
\noindent It is well known that mixed monotone operators were initially considered by Guo and Lakshmikantham \cite{GL87}. Thereafter, different authors considered the problem of existence of a fixed point for such mappings in Banach spaces and then in partially ordered metric spaces, see for instance \cite{G88,CZ13}.  The mixed monotone operator equation is important for applications due to the existence of particular classes of  integro-differential equations and boundary value problems that are solved by such equations \cite{HS}.

\medskip
\noindent The aim of this paper is two folds: first define the mixed $G$-monotone for both single and multivalued mappings,
 second extend the conclusion of Theorem \ref{ran} to both cases in metric spaces endowed with a graph.

\section{Preliminaries}

\noindent Let $G$ be a directed graph (digraph) with set of vertices $V(G)$ and set of edges $E(G)$ contains all the loops, i.e. $(x,x) \in E(G)$ for any $x \in V(G)$. Such digraphs are called reflexive. We also assume that $G$ has no parallel edges (arcs) and so we can identify $G$ with the pair $(V(G),E(G))$. By $G^{-1}$ we denote the conversion of a graph $G$, i.e., the graph obtained from $G$ by reversing the direction of edges. The letter $\widetilde{G}$ denotes the undirected graph obtained from $G$ by ignoring the direction of edges. Actually, it will be more convenient for us to treat $\widetilde{G}$ as a directed graph for which the set of its edges is symmetric. Under this convention,
$$E(\widetilde{G})=E(G)\bigcup E(G^{-1}).$$

\noindent If $x$ and $y$ are vertices in a graph $G$, then a (directed) path in $G$ from $x$ to $y$ of length $N$  is a sequence $(x_{i})_{i=0}^{i=N}$ of $N + 1$ vertices such that $x_{0} = x$, $x_{N} = y$ and $(x_{n-1},x_{n})\in E(G)$ for $i = 1,...,N$. A graph $G$ is connected if there is a directed path between any two vertices. $G$ is weakly connected if $\widetilde{G}$ is connected.

\medskip

\noindent In the sequel, we assume that $(X,d)$ is a metric space, and $G$ is a reflexive digraph (digraph) with set of vertices $V(G)=X$ and set of edges $E(G)$.\\

\begin{definition}\label{G-mapping}  Let $(X,d, G)$ be as described above.
\begin{itemize}
\item[(i)]  We say that a mapping $F: X\times X \rightarrow X$ has the mixed $G$-monotone property if
$$(x_1,x_2)\in E(G) \Longrightarrow (F(x_1,y),F(x_2,y))\in E(G),$$
for all $x_1, x_2, y \in X$, and
$$(y_1,y_2)\in E(G) \Longrightarrow (F(x,y_2),F(x,y_1))\in E(G),$$
for all $x, y_1, y_2 \in X.$
\item[(ii)] The pair $(x,y) \in X\times X$ is called a coupled fixed point of $F: X\times X \rightarrow X$ if
$$F(x,y) = x,\;\; and\; F(y,x) = y.$$
\end{itemize}
\end{definition}

\section{Main Results}

\noindent We begin with the extension of the main results of \cite{BL} to the case of metric spaces endowed with a graph.  Note that if $G$ is a directed graph defined on $X$ as described before, one can construct another graph on $X\times X$, still denoted by $G$, by
$$\Big((x,y), (u,v)\Big) \in E(G) \Longleftrightarrow (x,u) \in E(G)\; and\; (v,y) \in E(G),$$
for any $(x,y), (u,v) \in X\times X$.

\medskip

\begin{theorem}\label{BL}  Let $(X,d, G)$ be as above.  Assume that $(X,d)$ is a complete metric space.  Let $F: X\times X \rightarrow X$ be a continuous mapping having the mixed $G$-monotone property on $X$.  Assume there exists $k < 1$ such that
\begin{equation}\tag{BL}
d(F(x,y), F(u,v)) \leq \frac{k}{2} \Big[d(x,u) + d(y,v)\Big],
\end{equation}
for any $(x,y), (u,v) \in X\times X$ such that $\Big((x,y), (u,v)\Big) \in E(G)$.  If there exist $x_0, y_0 \in X$ such that $\Big((x_0,y_0), (F(x_0,y_0),F(y_0,x_0))\Big) \in E(G)$, then there exists $(x,y)$ a coupled fixed point of $F$ , i.e. $F(x,y) =x$ and $F(y,x) = y$.
\end{theorem}
\begin{proof} By assumption, there exist $x_0, y_0 \in X$ such that
$$(x_0, F(x_0,y_0)) \in E(G)\;\; and\; (F(y_0, x_0), y_0) \in E(G).$$
Set $x_1 = F(x_0,y_0)$ and $y_1 = F(y_0,x_0)$.  Then $(x_0,x_1) \in E(G)$ and $(y_1,y_0) \in E(G)$, which implies
$$d\Big(F(x_0,y_0), F(x_1,y_1)\Big) \leq \frac{k}{2} \Big[d(x_0,x_1) + d(y_0,y_1)\Big],$$
and
$$d\Big(F(y_1,x_1), F(y_0,x_0)\Big) \leq \frac{k}{2} \Big[d(x_0,x_1) + d(y_0,y_1)\Big].$$
By induction, we construct two sequences $\{x_n\}$ and $\{y_n\}$ in $X$ such that
\begin{enumerate}
\item[(i)] $x_{n+1} = F(x_n,y_n)$, and $y_{n+1} = F(y_n,x_n)$;
\item[(ii)] $\displaystyle d(x_n, x_{n+1}) \leq \frac{k}{2} \Big[d(x_{n-1},x_{n}) + d(y_{n-1},y_{n})\Big],$
\item[(iii)] $\displaystyle d(y_n, y_{n+1}) \leq \frac{k}{2} \Big[d(x_{n-1},x_{n}) + d(y_{n-1},y_{n})\Big],$
\end{enumerate}
for any $n \geq 1$.  From (ii) and (iii), we get
$$d(x_n, x_{n+1}) + d(y_n, y_{n+1}) \leq k\ \Big[d(x_{n-1},x_{n}) + d(y_{n-1},y_{n})\Big],$$
for any $n \geq 1$.  Therefore, we must have
$$d(x_n, x_{n+1}) + d(y_n, y_{n+1}) \leq k^n\ \Big[d(x_{0},x_{1}) + d(y_0,y_{1})\Big],$$
for any $n \geq 0$.  Hence from (ii), we get
$$d(x_n, x_{n+1}) \leq \frac{k}{2} \Big[d(x_{n-1},x_{n}) + d(y_{n-1},y_{n})\Big] \leq \frac{k}{2}\ k^{n-1} \ \Big[d(x_{0},x_{1}) + d(y_0,y_{1})\Big],$$
i.e., $\displaystyle d(x_n, x_{n+1}) \leq \frac{k^n}{2} \Big[d(x_0,x_{1}) + d(y_0,y_{1})\Big]$, for any $n \geq 0$.  Similarly, we will get
$$d(y_n, y_{n+1}) \leq \frac{k^n}{2} \Big[d(x_0,x_{1}) + d(y_0,y_{1})\Big],$$
for any $n \geq 0$.  Since $k < 1$, we conclude that $\sum d(x_n,x_{n+1})$ and $\sum d(y_n,y_{n+1})$ are convergent which imply that $\{x_n\}$ and $\{y_n\}$ are Cauchy sequences.  Since $(X,d)$ is complete, there exist $x, y \in X$ such that
$$\lim\limits_{n\rightarrow +\infty} x_n = x \; and\; \lim\limits_{n\rightarrow +\infty} y_n = y.$$
Since $F$ is continuous, we get from (i) above
$$x = \lim\limits_{n\rightarrow +\infty} x_{n+1} = \lim\limits_{n\rightarrow +\infty} F(x_n,y_n) = F(\lim\limits_{n\rightarrow +\infty} x_n, \lim\limits_{n\rightarrow +\infty} y_n) = F(x,y)$$
and similarly $y = F(y,x)$, i.e., $(x,y)$ is a coupled fixed point of $F$.
\end{proof}

\begin{example}
Let $X = \mathbb{R}$, $d(x, y) = |x - y|$ and $F : X \times X \rightarrow X$ be
defined by $F(x,y)=\frac{x+y}{5}$, $(x,y)\in X \times X$. Let $G$ be the reflexive digraph defined on  $X$ with $((x,y),(u,v))\in E(G)$ if and only if $x\leq u$ and $v\leq y$. Then $F$ is mixed $G$-monotone and satisfies condition (BL). Indeed, let $k=\frac{2}{3}$ then
$d(F(x,y), F(u,v))=|\frac{(x+y)}{5}-\frac{(u+v)}{5}|=|\frac{(x-u)}{5}+\frac{(y-v)}{5}|\leq \frac{1}{5}(|x-u|+|y-v|)\leq \frac{1}{3}(|x-u|+|y-v|)= \frac{2/3}{2} \Big[d(x,u) + d(y,v)\Big]$,
for any $(x,y), (u,v) \in X\times X$ such that $\Big((x,y), (u,v)\Big) \in E(G)$. Notice that $((0,0),(0,0))\in E(G)$. So by Theorem \ref{BL} we have that $F$ has a coupled fixed point $(0, 0)$. To illustrate the proof of Theorem \ref{BL}, let us consider $(x_0,y_0)=(0,1)$, $F(0,1)=F(1,0)=\frac{1}{5}$ (notice that $((0,1),(\frac{1}{5},\frac{1}{5}))\in E(G)$). Then $x_n=y_n=\frac{1}{5}(\frac{2}{5})^{n-1}\rightarrow 0$ as $n\rightarrow \infty$. Thus by Theorem \ref{BL} $(0,0)$ is a couple fixed point of $F$.
\end{example}

\noindent The continuity assumption of $F$ may be relaxed as it was done by Nieto et al \cite{NPRL}.  Indeed, we will say that $(X,d,G)$ has property $\textbf{(*)}$ if the following hold:
\begin{enumerate}
\item[(i)] for any $\{x_n\}$ in $X$ such that $(x_n, x_{n+1}) \in E(G)$ and $\lim\limits_{n \rightarrow +\infty} x_n = x$, then $(x_n,x) \in E(G)$, and
\item[(ii)] for any $\{x_n\}$ in $X$ such that $(x_{n+1}, x_{n}) \in E(G)$ and $\lim\limits_{n \rightarrow +\infty} x_n = x$, then $(x,x_n) \in E(G)$.
\end{enumerate}

\medskip

\noindent We have the following result.

\begin{theorem}\label{Nieto-BL} Let $(X,d, G)$ be as above.  Assume that $(X,d)$ is a complete metric space and $(X,d,G)$ has property $\textbf{(*)}$.  Let $F: X\times X \rightarrow X$ be a mapping having the mixed $G$-monotone property on $X$.  Assume there exists $k < 1$ such that
\begin{equation}\tag{BL}
d(F(x,y), F(u,v)) \leq \frac{k}{2} \Big[d(x,u) + d(y,v)\Big],
\end{equation}
for any $(x,y), (u,v) \in X\times X$ such that $\Big((x,y), (u,v)\Big) \in E(G)$.  If there exist $x_0, y_0 \in X$ such that $\Big((x_0,y_0), (F(x_0,y_0),F(y_0,x_0))\Big) \in E(G)$, then there exist $(x,y)$ a coupled fixed point of $F$.
\end{theorem}
\begin{proof} As we did in the proof of Theorem \ref{BL}, we construct $\{x_n\}$ and $\{y_n\}$ in $X$ such that
\begin{enumerate}
\item[(i)] $x_{n+1} = F(x_n,y_n)$, and $y_{n+1} = F(y_n,x_n)$;
\item[(ii)] $(x_n,x_{n+1}) \in E(G)$ and $(y_{n+1}, y_n) \in E(G)$;
\item[(iii)] $\displaystyle d(x_n, x_{n+1}) \leq \frac{k}{2} \Big[d(x_n,x_{n+1}) + d(y_n,y_{n+1})\Big],$
\item[(iv)] $\displaystyle d(y_n, y_{n+1}) \leq \frac{k}{2} \Big[d(x_n,x_{n+1}) + d(y_n,y_{n+1})\Big],$
\end{enumerate}
for any $n \geq 0$. Similar to the proof of Theorem \ref{BL}, we conclude that $\{x_n\}$ and $\{y_n\}$ are Cauchy.   Since $(X,d)$ is complete, then there exist $x, y \in X$ such that
$$\lim\limits_{n\rightarrow +\infty} x_n = x \; and\; \lim\limits_{n\rightarrow +\infty} y_n = y.$$
The property $\textbf{(*)}$ implies
$$(x_n,x) \in E(G)\; and\; (y,y_n) \in E(G), $$
for any $n \geq 0$.  Since $F$ has the mixed $G$-monotone property on $X$, we get
$$d\Big(F(x_n,y_n), F(x,y)\Big) \leq \frac{k}{2} \Big[d(x_n,x) + d(y_n,y)\Big],$$
and
$$d\Big(F(y_n,x_n), F(y,x)\Big) \leq \frac{k}{2} \Big[d(x_n,x) + d(y_n,y)\Big],$$
for any $n \geq 0$.  Hence
$$d\Big(x_{n+1}, F(x,y)\Big) \leq \frac{k}{2} \Big[d(x_n,x) + d(y_n,y)\Big],$$
and
$$d\Big(y_{n+1}, F(y,x)\Big) \leq \frac{k}{2} \Big[d(x_n,x) + d(y_n,y)\Big],$$
for any $n \geq 0$.  This imply
$$\lim\limits_{n\rightarrow +\infty} x_n = F(x,y) \; and\; \lim\limits_{n\rightarrow +\infty} y_n = F(y,x),$$
i.e., $F(x,y) = x$ and $F(y,x) = y$.
\end{proof}

\medskip

\noindent Under the assumptions of both Theorems \ref{BL} and \ref{Nieto-BL}, if assume that $(x_0,y_0) \in E(G)$, then we have $x = y$.  Indeed, it is easy to see that for any $u, v \in X$ such that $(u,v) \in E(G)$, then the condition (BL) implies
$$d(F(u,v),F(v,u)) \leq k\ d(u,v).$$
This will imply that $d(x_{n+1},y_{n+1}) \leq k\ d(x_n,y_n)$, for any $n \geq 0$.  In particular, we have $d(x_n,y_n) \leq k^n \ d(x_0,y_0)$, for ay $n \geq 0$.  Since $k < 1$, we conclude that $d(x,y) = \lim\limits_{n\rightarrow +\infty} d(x_n,y_n) = 0$, i.e., $x = y$.

\begin{remark}\label{weakly} In this remark, we discuss the uniqueness of the coupled fixed point.  Under the assumptions of both Theorems \ref{BL} and \ref{Nieto-BL}, let $(x,y)$ and $(u,v)$ be two coupled fixed points of $F$.  Assume that $\Big((x,y), (u,v)\Big) \in E(G)$.  Since $F$ has the mixed $G$-monotone property on $X$, we get
$$d(F(x,y), F(u,v)) \leq \frac{k}{2} \Big[d(x,u) + d(y,v)\Big],$$
and
$$d(F(v,u), F(y,x)) \leq \frac{k}{2} \Big[d(x,u) + d(y,v)\Big],$$
with $k < 1$.  Since $(x,y)$ and $(u,v)$ are coupled fixed points of $F$, we get
$$d(x,u) \leq \frac{k}{2} \Big[d(x,u) + d(y,v)\Big],\; and\;\; d(y,v) \leq \frac{k}{2} \Big[d(x,u) + d(y,v)\Big],$$
which implies
$$d(x,u) + d(y,v) \leq k\ \big(d(x,u) + d(y,v)\Big).$$
Hence $\displaystyle d(x,u) + d(y,v) = 0$, which yields $(x,y) = (u,v)$.  Moreover assume that there exist $x_0,y_0 \in X$ such that $\Big((x_0,y_0), (F(x_0,y_0),F(y_0,x_0))\Big) \in E(G)$.  Let $(u,v)$ be a coupled fixed point of $F$ such that $\Big((x_0,y_0), (u,v)\Big) \in E(G)$, then
$$d(F(x_0,y_0), F(u,v)) = d(F(x_0,y_0), u) \leq \frac{k}{2} \Big[d(x_0,u) + d(y_0,v)\Big],$$
and
$$d(F(v,u), F(y_0,x_0)) = d(v, F(y_0,x_0)) \leq \frac{k}{2} \Big[d(x_0,u) + d(y_0,v)\Big],$$
since $F$ has the mixed $G$-monotone property.  If $\{x_n\}$ and $\{y_n\}$ are the two sequences generated by $x_0, y_0, F$ in the proof of both Theorems \ref{BL} and \ref{Nieto-BL}, then we have
$$d(x_{n+1}, u) \leq \frac{k}{2} \Big[d(x_n,u) + d(y_n,v)\Big] \leq \frac{k^n}{2} \Big[d(x_0,u) + d(y_0,v)\Big],$$
and
$$d(v, y_{n+1}) \leq \frac{k}{2} \Big[d(x_n,u) + d(y_n,v)\Big] \leq \frac{k^n}{2} \Big[d(x_0,u) + d(y_0,v)\Big],$$
for any $n \geq 1$.  Since $k < 1$, we get
$$\lim\limits_{n\rightarrow +\infty} x_n = u \; and\; \lim\limits_{n\rightarrow +\infty} y_n = v.$$
Therefore given $x_0,y_0 \in X$ such that $\Big((x_0,y_0), (F(x_0,y_0),F(y_0,x_0))\Big) \in E(G)$, there exists a unique coupled fixed point $(x,y)$ of $F$ such that $\Big((x_0,y_0), (x,y)\Big) \in E(G)$.
\end{remark}

\medskip

\noindent In the next section we discuss the multivalued version of the main results of this section.

\section{Coupled Fixed Points of Multivalued monotone mappings}
\noindent Let $(X,d)$ be a metric space. We denote by ${\cal CB}(X)$ the collection of all nonempty closed and bounded subsets of $X$.  The Pompeiu-Hausdorff distance on $\cal{CB} (X)$ is defined by
$$ H(A,B):= \max \{\sup\limits_{b\in B} d(b,A), \sup\limits_{a\in A} d(a,B)\},$$
for $A,B\in \cal{CB} (X)$, where $d(a,B):= \inf\limits_{b\in B}d(a,b)$. Let $F:X\times X \rightarrow {\cal CB}(X)$ be a multivalued mapping. We will say that $F$ is continuous if for any sequences $\{x_n\}$ and $\{y_n\}$ which converge respectively to $x$ and $y$, we have
$$\lim\limits_{n\rightarrow \infty} H(F(x_n,y_n),F(x,y))=0.$$
\noindent The following technical result is useful to explain our definition later on.

\begin{lemma}\label{HHH} Let $(X,d)$ be a metric space.  For any $A, B \in \cal{CB}(X)$ and $\varepsilon > 0$, we have:
\begin{enumerate}
\item[(i)] for $a \in A$, there exists $b \in B$ such that
$$d(a,b) \leq H(A,B) + \varepsilon;$$
\item[(ii)] for $b \in B$, there exists $a \in A$ such that
$$d(a,b) \leq H(A,B) + \varepsilon.$$
\end{enumerate}
\end{lemma}

\medskip\noindent Note that from Lemma \ref{HHH}, whenever one uses multivalued mappings which involves the Pompeiu-Hausdorff distance, then one must assume that the multivalued mappings have bounded values.  Otherwise, one has only to assume that the multivalued mappings have nonempty closed values.

\medskip\medskip

\noindent Let $(X,d,G)$ be as before. We denote by ${\cal C}(X)$ the collection of all nonempty closed subsets of $X$. Let $F:X\times X \rightarrow {\cal C}(X)$ be a multivalued mapping. We will say that $F$ has the mixed $G$-monotone property on $X$ if:
\begin{enumerate}
\item[(i)] for any $x_1, x_2, y \in X$ such that $(x_1,x_2)\in E(G)$, for any $u \in F(x_1,y)$, there exists $v \in F(x_2,y)$ such that $(u,v) \in E(G)$;
\item[(ii)] for any $x, y_1, y_2 \in X$ such that $(y_1,y_2)\in E(G)$, for any $u \in F(x,y_2)$, there exists $v \in F(x,y_1)$ such that $(u,v) \in E(G)$;
\end{enumerate}
\noindent The pair $(x,y) \in X\times X$ is called a coupled fixed point of $F: X\times X \rightarrow {\cal C}(X)$ if
$$ x \in F(x,y),\;\; and\; y \in F(y,x).$$
\noindent The multivalued version of the condition (BL) may be stated as

\begin{definition} The multivalued mapping $F: X\times X \rightarrow {\cal C}(X)$ is said to satisfy the condition (MBL) if there exists $k < 1$ such that for any $(x,y), (u,v) \in X\times X$ with $\Big((x,y), (u,v)\Big) \in E(G)$, and for any $a \in F(x,y)$ there exists $b \in F(u,v)$ such that
\begin{equation}\tag{MBL}
d(a, b) \leq \frac{k}{2} \Big[d(x,u) + d(y,v)\Big].
\end{equation}
\end{definition}

\noindent Next we give an analogue result of Theorem \ref{BL} to the case of mixed $G$-monotone multivalued mappings in metric spaces.
\begin{theorem}\label{MBL}  Let $(X,d, G)$ be as above.  Assume that $(X,d)$ is a complete metric space.  Let $F: X\times X \rightarrow \cal{CB}(X)$ be a continuous multivalued mapping having the mixed $G$-monotone property on $X$ and satisfying (MBL) condition. If there exist $x_0, y_0 \in X$ and $x_1\in F(x_0,y_0)$, $y_1 \in F(y_0,x_0)$ such that $\Big((x_0,y_0), (x_1,y_1)\Big) \in E(G)$, then there exists $(x,y)$ a coupled fixed point of $F$.
\end{theorem}
\begin{proof} By assumption, there exist $x_0, y_0 \in X$ and $x_1\in F(x_0,y_0)$, $y_1 \in F(y_0,x_0)$ such that $\Big((x_0,y_0), (x_1,y_1)\Big) \in E(G)$.  Then $(x_0,x_1) \in E(G)$ and $(y_1,y_0) \in E(G)$. Since $F$ satisfies the (MBL) condition, then there exists $x_2\in F(x_1,y_1)$ and $y_2\in F(y_1,x_1)$ with
$$d\Big(x_1, x_2\Big) \leq \frac{k}{2} \Big[d(x_0,x_1) + d(y_0,y_1)\Big],$$
and
$$d\Big(y_1, y_2\Big) \leq \frac{k}{2} \Big[d(x_0,x_1) + d(y_0,y_1)\Big].$$
By induction, we construct two sequences $\{x_n\}$ and $\{y_n\}$ in $X$ such that
\begin{enumerate}
\item[(i)] $x_{n+1} \in F(x_n,y_n)$, and $y_{n+1} \in F(y_n,x_n)$;
\item[(ii)] $\displaystyle d(x_n, x_{n+1}) \leq \frac{k}{2} \Big[d(x_{n-1},x_{n}) + d(y_{n-1},y_{n})\Big],$
\item[(iii)] $\displaystyle d(y_n, y_{n+1}) \leq \frac{k}{2} \Big[d(x_{n-1},x_{n}) + d(y_{n-1},y_{n})\Big],$
\end{enumerate}
for any $n \geq 1$.  As we did in the proof of Theorem \ref{BL}, we have
$$d(x_n, x_{n+1}) \leq \frac{k^n}{2} \Big[d(x_0,x_{1}) + d(y_0,y_{1})\Big],$$
and
$$d(y_n, y_{n+1}) \leq \frac{k^n}{2} \Big[d(x_0,x_{1}) + d(y_0,y_{1})\Big],$$
for any $n \geq 1$.  Since $k< 1$ we conclude that $\sum d(x_n,x_{n+1})$ and $\sum d(y_n,y_{n+1})$ are convergent which imply that $\{x_n\}$ and $\{y_n\}$ are Cauchy sequences.  Since $(X,d)$ is complete, then there exist $x, y \in X$ such that
$$\lim\limits_{n\rightarrow +\infty} x_n = x \; and\; \lim\limits_{n\rightarrow +\infty} y_n = y.$$
Since $F$ is continuous, we get
$$\lim\limits_{n \rightarrow \infty} H(F(x_n,y_n), F(x,y)) = 0.$$
Since $x_{n+1} \in F(x_n,y_n)$, Lemma \ref{HHH} implies the existence of $b_n \in F(x,y)$ such that
$$d(x_{n+1}, b_n) \leq H(F(x_n,y_n), F(x,y)) + \frac{1}{n},$$
for any $n \geq 1$.  Clearly, we have $\lim\limits_{n \rightarrow \infty} b_n = x$.  Since $F(x,y)$ is closed, we conclude that $x \in F(x,y)$.  Similarly, we will show that $y \in F(y,x)$, i.e., $(x,y)$ is a coupled fixed point of $F$.
\end{proof}

\begin{example}
Let $X = \mathbb{R}$, $d(x, y) = |x - y|$ and $F : X \times X \rightarrow \cal{CB}(X)$ be
defined by $F(x,y)=\{ -\frac{x+y}{5}, \frac{x+y}{5}\}$, $(x,y)\in X \times X$. Let $G$ be the reflexive digraph defined on  $X$ with $((x,y),(u,v))\in E(G)$ if and only if $x\leq u$ and $v\leq y$. Then $F$ is mixed $G$-monotone and satisfies condition (MBL). Indeed, let $k=\frac{2}{3}$ and for any  $u \in F(x.y)$ take $v=u\in F(y,x)$,then
$0=d(u, v)\leq \frac{1}{5}(|u-x|+|v-y|)\leq \frac{1}{3}(|u-x|+|v-y|)= \frac{2/3}{2} \Big[d(u,x) + d(v,y)\Big]$, for any $(x,y), (u,v) \in X\times X$  with $\Big((x,y), (u,v)\Big) \in E(G)$. Notice that $((0,0),(0,0))\in E(G)$. So by Theorem \ref{MBL} we have that $F$ has a coupled fixed point $(0, 0)$. To illustrate the proof of Theorem \ref{MBL}, let us consider $(x_0,y_0)=(0,1)$, if $u=\frac{-1}{5}\in F(0,1)$ take $v=\frac{-1}{5}$(notice that $((0,1),(\frac{-1}{5},\frac{-1}{5}))\in E(G)$). Then $x_n=y_n=\frac{-1}{5}(\frac{2}{5})^{n-1}\rightarrow 0$ as $n\rightarrow \infty$. Thus by Theorem \ref{MBL} $(0,0)$ is a couple fixed point of $F$.
\end{example}

\noindent As we did in the single valued case, the continuity assumption of $F$ can be relaxed using Property $\textbf{(*)}$.  We have the following result.

\begin{theorem}\label{Nieto-MBL} Let $(X,d, G)$ be as above.  Assume that $(X,d)$ is a complete metric space and $(X,d,G)$ has property $\textbf{(*)}$.  Let $F: X\times X \rightarrow \cal{C}(X)$ be a multivalued mapping having the mixed $G$-monotone property on $X$ and satisfying (MBL) condition.   If there exist $x_0, y_0 \in X$ and $x_1\in F(x_0,y_0)$, $y_1 \in F(y_0,x_0)$ such that $\Big((x_0,y_0), (x_1,y_1)\Big) \in E(G)$, then there exist $(x,y)$ a coupled fixed point of $F$.
\end{theorem}
\begin{proof} As we did in the proof of Theorem \ref{BL}, we construct $\{x_n\}$ and $\{y_n\}$ in $X$ such that
\begin{enumerate}
\item[(i)] $x_{n+1} \in F(x_n,y_n)$, and $y_{n+1} \in F(y_n,x_n)$;
\item[(ii)] $(x_n,x_{n+1}) \in E(G)$ and $(y_{n+1}, y_n) \in E(G)$;
\item[(iii)] $\displaystyle d(x_n, x_{n+1}) \leq \frac{k}{2} \Big[d(x_{n-1},x_{n}) + d(y_{n-1},y_{n})\Big],$
\item[(iv)] $\displaystyle d(y_n, y_{n+1}) \leq \frac{k}{2} \Big[d(x_{n-1},x_{n}) + d(y_{n-1},y_{n})\Big],$
\end{enumerate}
for any $n \geq 1$. Clearly both sequences $\{x_n\}$ and $\{y_n\}$ are Cauchy.   Since $(X,d)$ is complete, then there exist $x, y \in X$ such that
$$\lim\limits_{n\rightarrow +\infty} x_n = x \; and\; \lim\limits_{n\rightarrow +\infty} y_n = y.$$
The property $\textbf{(*)}$ implies
$$(x_n,x) \in E(G)\; and\; (y,y_n) \in E(G), $$
for any $n \geq 1$.  Since $F$ has the mixed $G$-monotone property on $X$, there exist $x_n^*\in F(x,y)$ and $y_n^*\in F(y,x)$ with
$$d(x_{n+1}, x_n^*) \leq \frac{k}{2} \Big[d(x_n,x) + d(y_n,y)\Big],$$
and
$$d(y_{n+1}, y_n^*) \leq \frac{k}{2} \Big[d(x_n,x) + d(y_n,y)\Big],$$
for any $n \geq 1$.  This will imply
$$\lim\limits_{n\rightarrow +\infty} d(x_{n+1}, x_n^*) = 0 \;\; \text{and}\;\; \lim\limits_{n\rightarrow +\infty} d(y_{n+1}, y_n^*) = 0.$$
Therefore, we have
$$\lim\limits_{n\rightarrow +\infty} x_n^* = x \; \text{and} \; \lim\limits_{n\rightarrow +\infty} y_n^* = y.$$
Since $F(x,y)$ and $F(y,x)$ are closed, we conclude that $x\in F(x,y)$ and $y\in F(y,x)$, i.e., $(x,y)$ is a coupled fixed point of $F$.
\end{proof}

\medskip\medskip\medskip

\noindent {\large{\textbf{Acknowledgements}}}
\medskip \\
\noindent The authors would like to acknowledge the support provided by the Deanship of Scientific Research at King Fahd University of Petroleum \& Minerals for funding this work through project No. IP142-MATH-111.

\end{document}